\documentclass[reqno]{amsart}

\usepackage{amsthm, amssymb, latexsym, tabls, hhline,lscape}
\newtheorem{theorem}{Theorem}[section]
\newtheorem{proposition}[theorem]{Proposition}
\newtheorem{corollary}[theorem]{Corollary}
\newtheorem{lemma}[theorem]{Lemma}
\newtheorem{conj}[theorem]{Conjecture}
\theoremstyle{definition}

\newtheorem{example}[theorem]{Example}

\theoremstyle{remark}
\newtheorem{remark}[theorem]{Remark}

\numberwithin{equation}{section}

\newcommand{\QSym}{\ensuremath{\mathit{QSym}}}

\newcommand{\A}{\mathcal{A}}

\newcommand{\shuf}{{\,\sqcup\!\sqcup\,}}

\makeindex

\title[Non-commutative Combinatorial Inverse Systems]
{\bf  Non-commutative Combinatorial Inverse Systems}

\author{J.-C.~Aval}\address[Jean-Christophe Aval]
{Labri\\ Universit\'e Bordeaux 1\\ 351 cours de la
Lib\'eration\\ 33405 Talence cedex\\ FRANCE}
\email{aval@labri.fr}
\urladdr{http://www.labri.fr/perso/aval}

\author{N. Bergeron}\address[Nantel Bergeron]
{Department of Mathematics and Statistics\\ York  University\\ To\-ron\-to, Ontario M3J 1P3\\ CANADA}
\email{bergeron@mathstat.yorku.ca}
\urladdr{http://www.math.yorku.ca/bergeron}

 \author{H. Li}\address[Huilan Li]
 {Department of Mathematics\\ Drexel  University\\ Philadelphia, PA 19104\\ U.S.A}
 \email{huilan.li@gmail.com}
 \urladdr{http://www.math.drexel.edu/\~{}huilan/}

\date{\today}

\thanks{J.-C. Aval is supported in part by the ANR project MARS (BLAN06-2$\_$0193)}
\thanks{N. Bergeron is supported in part by CRC and NSERC}
\thanks{H. Li is supported in part by CRC, NSERC and NSF grant DMS-0652641}

\keywords{}

\subjclass[2000]{}

\begin{document}
\maketitle

%%%%%%%%%%%%%%%%%%%%%%%%%%%%%%%%%%%
%%%%%%%%%%%%%%%%%%%%%%%%%%%%%%%%%%%
\begin{abstract}
We introduce the notion of a combinatorial inverse system in non-commutative variables.
We present two important examples, some conjectures and results.
These conjectures and results were suggested and supported by computer investigations.
\end{abstract}

%%%%%%%%%%%%%%%%%%%%%%%%%%%%%%%%%%%
%%%%%%%%%%%%%%%%%%%%%%%%%%%%%%%%%%%
\section{Introduction}\setcounter{equation}{0}

Inverse systems \cite{IS} and Gr\"obner bases \cite{G} are very useful tools to study finitely generated commutative algebras. 
In practice we are given a presentation of an algebra with generators and relations. That is, the algebra is the quotient 
of a free commutative algebra (a polynomial ring in finitely many variables) by the ideal of relations. Inverse systems and Gr\"obner bases 
allow one to give explicit linear bases and to extract all the properties of the quotient.

We are interested in the case where the  ideals of relations are obtained from a family of algebras related to  a combinatorial construct. More precisely, a \textit{combinatorial inverse system} is a family of inverse systems obtained from a family of ideals $\{I_n\}$  where $I_n$ is generated by a combinatorial Hopf algebra \cite{ABS} restricted to $n$ variables. There are several important examples of combinatorial inverse systems (see \cite{ABB, AB,  GP,Haim}). Coinvariants of the symmetric groups are particularly well studied \cite{Ch,St} and the spaces of diagonal harmonics \cite{Haim} are still intensively studied.

For finitely generated non-commutative algebras, much less is known. In contrast with the commutative case, the ideal of relations is not guaranteed to be finitely generated. This may cause many problems. In particular, the problem of finding a Gr\"obner basis \cite{G} is not decidable. A Gr\"obner basis in non-commutative variables is in general infinite and the Buchberger algorithm may not stop. In the non-commutative setting, even the inverse system associated to the symmetric group invariants is not well understood. 

In this paper we introduce the basic notion of  a non-commutative inverse system in Section~\ref{sec:NCIS}. We only consider homogeneous ideals since combinatorial inverse systems are generated by homogenous elements. In Section~\ref{sec:CIS}, we define more precisely the notion of a combinatorial inverse system and give some examples in commutative variables.  We then return our attention to non-commutative combinatorial inverse systems. Section~\ref{sec:SymIS} is dedicated to the non-commutative combinatorial inverse system for the symmetric polynomials in $n$ non-commutative variables (symmetric group invariants). It is not known whether this system is finite for all $n$. We have computed this system for $n=1,2,3$ and $4$ and conjecture it is finite for all $n$. 
This conjecture would guarantee that the non-commutative Buchberger algorithm stops and that the ideals generated by non-commutative symmetric polynomials  always have decidable finite Gr\"obner bases.
We give some evidence of this and some weaker conjectures. In Section~\ref{sec:QSymIS} we present the non-commutative combinatorial inverse system of Quasi-symmetric functions (Temperley-Lieb algebras invariants). In recent work we have shown that this inverse system is finite, which gives us here further evidence for our conjectures in Section~\ref{sec:SymIS}.

%%%%%%%%%%%%%%%%%%%%%%%%%%%%%%%%%%%
%%%%%%%%%%%%%%%%%%%%%%%%%%%%%%%%%%%

\section{Elements of Non-commutative Inverse Systems}\label{sec:NCIS}

\subsection{Non-commutative inverse systems}

Let $R$ be the ring of polynomials in non-commuting variables  $\{x_1,x_2,\ldots, x_n\}$. That is
$$R=\mathbb{C}\langle x_1,x_2,\ldots, x_n\rangle.$$ For $a\in \{x_1,x_2,\ldots, x_n\}$, we define on $R$
the operator ${d}_a$ by
$${d}_a\cdot w=\left\{\begin{array}{ll}
                     u &\hbox{if }w=au,\\
                     0 &\hbox{otherwise,}
                    \end{array}
             \right.$$
where $w$ and $u$ are monomials. We refer to monomials as words and to its variables as letters. We think of this operator as a derivative although it does not
satisfy Leibniz's Rule. Let $u=u_1u_2\cdots u_k$ be a word of degree $k$ with
$u_j=x_{i_j}$ for some $1\leq i_j\leq n$. We denote 
$$\overleftarrow{u}=u_{k}u_{k-1}\cdots u_1$$ and 
$${d}_u\cdot w=(d_{u_1}\cdot (d_{u_2}\cdot (\cdots (d_{u_k}\cdot w)\cdots))).$$
A \textit{pairing} $<,>$ on $R$ is defined on the monomials (or words)
by
$$<u,v>=\delta_{u,v}=[{d}_{\overleftarrow{u}}\cdot v]_{c.t.},$$
where $[f]_{c.t.}$ means the constant term of $f$ and
$$\delta_{u,v}=\left\{\begin{array}{ll}
                     1 &\hbox{if }u=v,\\
                     0 &\hbox{otherwise.}
                    \end{array}
             \right.$$
For each $f\in R$, we define
$$f(d)=\sum_{w}c_wd_w,$$
when $f=\sum_wc_ww$, where $c_w\in\mathbb{C}$. Furthermore,
$$\overleftarrow{f}(d)=\sum_{w}c_wd_{\overleftarrow{w}} \mbox{ and }<f,P>=\big[\overleftarrow{f}(d)\cdot P\big]_{c.t.}.$$

Each $f\in R$ can be written as $f=f_n+f_{n-1}+\cdots+f_0$, where $f_i$ is the component in $f$ of degree $i$ for all $0\leq i\leq n$. We say that $I\subseteq R$ is a \textit{homogeneous ideal} if
$f=f_n+f_{n-1}+\cdots+f_0\in I$ implies $f_i\in I$ for all $0\leq
i\leq n$. In
fact, $I$ is a homogeneous ideal if and only if $I$ is generated by homogeneous elements of $R$.
%That is $I=<g_r: g_r\hbox{is homogeneous}>$.
For any homogeneous ideal $I$ we define
$$I^{\perp}=\{P\in R:\ <f,P>=0,\ \forall f\in I\}.$$ If $I$ is
homogeneous, clearly $R=I\oplus I^{\perp}.$ In this case, we can compute $ I^{\perp}$ independently in each (finite dimensional) homogeneous component of $R$.

\begin{lemma}\label{lem:eq} Let $I$ be a homogeneous ideal. Then
$$I^{\perp}=\{P\in R:\ \overleftarrow{f}(d)\cdot P=0,\ \forall f\in I\}.$$
This means that $I^\perp$ is the solution to a system of (differential)
equations.
\end{lemma}

\begin{proof}
Let $P\in R$ such that $\overleftarrow{f}(d)\cdot P=0$ for all $f\in
I$. Then $\big[\overleftarrow{f}(d)\cdot P\big]_{c.t.}=0$. So
$<f,P>=0$ for all $f\in I$. Hence $P\in I^{\perp}$. Conversely, let
$P\in I^{\perp}$ and $f\in I$. Then $fu\in I$ for any word $u$. So
$\big[{d}_{\overleftarrow{u}}\overleftarrow{f}(d)\cdot
P\big]_{c.t.}=0$ for all $u$. This implies that
$\overleftarrow{f}(d)\cdot P=0$.
\end{proof}

\begin{lemma}\label{lem:cone}
Let $I$ be a homogeneous  ideal. Then $I^{\perp}$ is closed
under derivation, i.e., ${d}_{u}\cdot P\in I^{\perp}$ for all $P\in
I^{\perp}$ and $u$.
\end{lemma}

\begin{proof}
Let $P\in I^{\perp}$. Then $<f,P>=0$ for all $f\in I$. For any
$f\in I$,
\begin{eqnarray*}
<f, {d}_u\cdot P>&=&\big[\overleftarrow{f}(d){d}_u\cdot
P\big]_{c.t.}\\
&=&<\overleftarrow{u}f,P>\\
&=&0, \hspace{2cm}\hbox {since }\overleftarrow{u}f\in I.
\end{eqnarray*}
Hence ${d}_u\cdot P \in I^{\perp}$.
\end{proof}

Since $R=I\oplus I^{\perp}$, $R/I\cong I^\perp$. Thanks to Lemma~\ref{lem:eq} and Lemma~\ref{lem:cone}, the space $I^\perp$ is the solution to a system of (differential) equations that is closed under differentiation. This is what we will refer to as the {\sl noncommutative inverse system} for the homogeneous ideal $I$. 

\begin{remark}
In some cases it may be interesting to study left (or right) ideals only. This still has some rich structure (see \cite{BRRZ,BRZ}), but the quotient is very likely to be infinite and has a different algebraic meaning. 
\end{remark}

\subsection{Relationship between commutative and non-commutative}

We are now interested in the relationship between polynomials in
commuting and non-commuting variables. Let $u=u_1u_2\cdots u_k$ be
a word of length $k$ with $u_j=x_{i_j}$ for some $1\leq i_j\leq
n$. Let $\sigma\in S_n$ and $\pi\in S_k$. We define
$$\sigma\circ u=x_{\sigma(i_1)}x_{\sigma(i_2)}\cdots x_{\sigma(i_k)}$$
and
$$u\circ \pi=u_{\pi(1)}u_{\pi(2)}\cdots u_{\pi(k)}.$$
For any nonnegative integral vector
$\alpha=(\alpha_1,\alpha_2,\ldots,\alpha_n)$, we define
$$x^{\alpha}=x_1^{\alpha_1}x_2^{\alpha_2}\cdots x_n^{\alpha_n}$$ and
$$\alpha!=\alpha_1!\alpha_2!\cdots \alpha_n!.$$
Consider the maps
$$\begin{array}{rcl}
\chi:\mathbb{C}\langle x_1,x_2,\ldots, x_n\rangle&
\longrightarrow& \mathbb{C}[x_1,x_2,\ldots, x_n]\\
x_i&\mapsto& x_i
\end{array}$$ and
$$\psi:\mathbb{C}[x_1,x_2,\ldots, x_n]\ \longrightarrow\ \mathbb{C}\langle x_1,x_2,\ldots, x_n\rangle$$
defined by
$$\psi(x^\alpha)=\sum_{\pi\in S_{k}}u\circ \pi,$$
where $k=\alpha_1+\alpha_2+\cdots\alpha_n$ and $u$ is any word such that $\chi(u)=x^\alpha$. 
This is well defined since the letters of $u$ in the definition of $\psi$ are permuted in all possible ways, hence it does not depend on the choice of $u$.
We know that $\chi$ is an algebra homomorphism. On
the other hand,  $\psi$ is   a linear injection which does not preserve products.

We define
$\partial_{x^{\alpha}}=\partial_1^{\alpha_1}\partial_2^{\alpha_2}\cdots
\partial_n^{\alpha_n},$ where $\partial_i$ is the partial
derivative operator over $x_i$, i.e., normally
$\frac{\partial}{\partial x_i}$, and the product of these
operators means the composition of them. For any $Q\in \chi(R)=\mathbb{C}[x_1,x_2,\ldots, x_n]$,
we define $Q(\partial)$ by replacing $x_i$ in $Q$ with $\partial_i$. 
For $P,Q\in \mathbb{C}[x_1,x_2,\ldots, x_n]$, the standard pairing is $<Q,P> =[Q(\partial)\cdot P]_{c.t}$. 
For an ideal $J\subseteq \chi(R)=\mathbb{C}[x_1,x_2,\ldots, x_n]$,
we define 
 $$\begin{array}{rl}
   J^\perp&=\ \{P\in  \chi(R):\ <Q,P>=0,\ \forall Q\in J\} \cr
 &=\  \{P\in  \chi(R):\ {Q}(\partial)\cdot P=0,\ \forall Q\in J\}.
 \end{array}$$
 This is the usual (commutative) inverse system for a (commutative) ideal.

\begin{lemma}\label{l3}
Let $P\in \chi(R)$ and $f \in R$. Then
$$\big[\overleftarrow{f}(d)\cdot \psi(P)\big]_{c.t.}=\big[\chi(f)(\partial)\cdot P\big]_{c.t.}.$$
\end{lemma}

\begin{proof}Let $u$ and $v$ be two words such that $\chi(u)=x^{\alpha}$ and $\chi(v)=x^{\beta}$. First, we show that
$$\big[{d}_{\overleftarrow{v}}\cdot \psi(u)\big]_{c.t.}=\alpha ! \delta_{\chi(u),\chi(v)}
=\big[\partial_{\chi(v)}\cdot u\big]_{c.t.}.$$
\begin{eqnarray*}
\big[{d}_{\overleftarrow{v}}\cdot \psi(u)\big]_{c.t.}&=&
\big[{d}_{\overleftarrow{v}}\cdot \sum_{\pi\in S_k}u\circ \pi\big]_{c.t.}\\
&=&\sum_{\pi\in S_k}\big[{d}_{\overleftarrow{v}}\cdot
(u\circ\pi)\big]_{c.t.}\\
&=&\sum_{\pi\in S_k}\delta_{v,u\circ\pi}\\
&=&\#\{\pi\in S_k:\ v_i=u_{\pi(i)}\}.
\end{eqnarray*}
Clearly, $\#\{\pi\in S_k:\ v_i=u_{\pi(i)}\}$ depends only on
$\chi(v)$ and $\chi(u)$. So
\begin{eqnarray*}
\big[{d}_{\overleftarrow{v}}\cdot \psi(u)\big]_{c.t.}&=&\#\{\pi\in
S_k:\
v=u\circ \pi\}\\
&=&\left\{\begin{array}{ll}
          0& \hbox{ if } \alpha\neq\beta\\
          \alpha!& \hbox{ otherwise}
          \end{array}
 \right.\\
&=&\alpha!\delta_{\chi(u),\ \chi(v)}\  = \  \alpha!\delta_{x^\alpha,\ x^\beta} .
\end{eqnarray*}

On the other side,
\begin{eqnarray*}
\big[\partial_{\chi(v)}\cdot u\big]_{c.t.}&=&\big[\partial_{\chi(v)}\cdot x^{\alpha}\big]_{c.t.}\\
&=&\alpha!\delta_{\alpha,\beta}.
\end{eqnarray*}
Since $\partial,\ d$ and $[\ ]_{c.t.}$ are linear maps, 
$$\big[{\overleftarrow{f}(d)}\cdot \psi(P)\big]_{c.t.}=\big[\chi(f)(\partial)\cdot P\big]_{c.t.}.$$
\end{proof}

\begin{corollary}\label{cor:inc}
Let $P\in\chi(R)$ and $I$ be any homogeneous ideal in
$R$. Then
$$P\in\chi(I)^{\perp}\Leftrightarrow\psi(P)\in I^{\perp}.$$
\end{corollary}

\begin{proof}
\begin{eqnarray*}
P\in\chi(I)^{\perp}&\Leftrightarrow&\forall f\in I,\ <\chi(f),P>=0\\
&\Leftrightarrow&\forall f\in I,\ \big[\chi(f)(\partial)\cdot P\big]_{c.t.}=0\hspace{1cm}\hbox{by definition}\\
&\Leftrightarrow&\forall f\in I,\ \big[{\overleftarrow{f}(d)}\cdot
\psi(P)\big]_{c.t.}=0\hspace{1cm}\hbox{from Lemma \ref{l3}}\\
&\Leftrightarrow&\forall f\in I,\ <f,\psi(P)>=0\hspace{1.5cm}\hbox{by definition}\\
&\Leftrightarrow&\psi(P)\in I^{\perp}.
\end{eqnarray*}
\end{proof}

Corollary~\ref{cor:inc} gives us a linear inclusion $\psi\colon \chi(I)^\perp\hookrightarrow I^\perp$ of inverse systems and a surjection of algebras
$\chi:R/I\rightarrow \chi(R)/\chi(I)$.

%%%%%%%%%%%%%%%%%%%%%%%%%%%%%%%%%%%
%%%%%%%%%%%%%%%%%%%%%%%%%%%%%%%%%%%

\section{Combinatorial Inverse Systems (commutative)}\label{sec:CIS}

A combinatorial Hopf algebra as defined in \cite{ABS} is a pair  $(\mathcal{H},\zeta)$ where $\mathcal H$ is a graded connected Hopf algebra and $\zeta\colon {\mathcal H}\to {\mathbb C}$ is an algebra morphism. The map $\zeta$ serves as a measure for some desired combinatorial invariants and will not be used here. For many examples, $\mathcal H$ is described with a homogeneous basis $\{b_\lambda\}$ such that all algebraic  structure constants are non-negative integers. In the commutative case ${\mathcal H}$ is realized as a subalgebra of ${\mathbb C}[\![X]\!]$, the homogeneous series in countably many variables $X$. Given this, we can restrict $\mathcal H$ to finitely many variables $X_n$ using an evaluation map ${\mathcal H}\hookrightarrow{\mathbb C}[\![ X]\!]\to{\mathbb C}[ X_n]$ where $x=0$ for  all $x\in X- X_n$. If $X_1\subset X_2\subset \cdots$ and $\lim_{n\to\infty} X_n=X$,
we obtain  a family of ideals $I_n=\langle b_\lambda(X_n) \rangle \subseteq {\mathbb C}[ X_n] $ where $b_\lambda(X_n)$ denotes the image of a basis element of $\mathcal H$ under the map ${\mathcal H}\to{\mathbb C}[ X_n]$ described above. We say that a family $\{I_n^\perp\}_{n\ge 0}$ obtained in this way is a \textit{combinatorial inverse system}.

To motivate our definition we present three key examples along with their main features.

\begin{example} \label{ex:sym}
{\sl Symmetric functions}:  The Hopf algebra of symmetric functions \cite{Ma} is $Sym={\mathbb C}[p_1,p_2,\ldots]$ where the comultiplication is given by $\Delta(p_k)=p_k\otimes 1 + 1 \otimes p_k$. The degree of $p_k$ is set to be $k$. There is an embeding 
$Sym\hookrightarrow {\mathbb C}[\![ X]\!]$ with $X=x_1,x_2,\ldots$, given by 
  $$p_k=\sum_{i\ge 1} x_{i} ^k.$$
The space $Sym$ is a combinatorial Hopf algebra satisfying the criteria above with $X_n=x_1,\ldots,x_n$. We can thus construct its combinatorial inverse system. That is, the inverse systems $H_n=\{I_n^\perp\}_{n\ge 1}$ corresponding to the ideals $I_n=\langle p_k(X_n) : k\ge 1\rangle \subseteq  {\mathbb C}[ X_n]$. These spaces are central in mathematics and  appear in a larger class of spaces in invariant and coinvariant theory \cite{Ch,St}. Since $H_n=\{P\in \mathbb{C}[X_n]: \phi(\partial) P=0, \forall \phi\in I_n\}$ it consists of polynomials $P$ that are solutions to the equation $p_2(\partial)P=0$, that is, harmonic polynomials. The spaces $H_n$ are also known as the symmetric harmonics in $n$ variables. They have been extensively studied in several contexts and satisfy very fundamental properties. 
For $n\ge 1$, some examples are

\begin{enumerate}
\item Let $\Delta_n=\prod_{1\le i<j\le n} (x_i-x_j)$ denote the VanDerMonde determinant. Then
  $$H_n=\big\{ P(\partial) \Delta_n : P\in  {\mathbb C}[ X_n] \big\}; $$

\item  The dimension of $H_n$  is $n!$;

\item The space $H_n$ is the left regular
representation of the symmetric group $S_n$;

\item As an $S_n$-module,
 $$H_n = \bigoplus_{k=0}^{n(n-1)\over 2} H_n^{(k)}$$ is graded. The $q$-Frobenius characteristic is
  $$ {\mathcal F}_q(H_n) = \sum_{k=0}^{n(n-1)\over 2} q^k char(H_n^{(k)}) = H_n(X;q)$$
where $char(H_n^{(k)})$ is the symmetric function associated with the representation $H_n^{(k)}$ and $H_n(X;q)$ is the Hall-Littlewood symmetric function \cite{Ma}.
\end{enumerate}

\end{example}

\begin{example}\label{ex:dsym}
 {\sl Diagonal invariants and Diagonal harmonics}: The Hopf algebra of MacMahon symmetric functions (also known as diagonal invariants) is $DSym={\mathbb C}[p_{a,b}: (a,b)\in \mathbb{Z}\times\mathbb{Z}]$ where $p_{0,0}=1$ and the comultiplication is given by 
$\Delta(p_{a,b})=p_{a,b}\otimes 1 + 1\otimes p_{a,b}.$
This space is bigraded and the degree of $p_{a,b}$ is set to be $(a,b)$. 
There is an embeding 
$Sym\hookrightarrow {\mathbb C}[\![ X;Y]\!]$ with $X;Y=x_1,x_2,\ldots,y_1,y_2,\ldots$, given by 
  $$p_{a,b}=\sum_{i\ge 1} x_i^a y_i^b.$$
Taking $X_n;Y_n=x_1,\ldots,x_n,y_1,\ldots y_n$ then $DSym$ is a combinatorial Hopf algebra satisfying the criteria above. We can thus construct its combinatorial inverse system. This is the inverse system $DH_n=\{I_n^\perp\}_{n\ge 1}$ corresponding to the ideals $I_n=\langle p_{a,b}(X_n;Y_n) : (a,b)\in \mathbb{Z}\times\mathbb{Z}\rangle \subseteq  {\mathbb C}[ X_n;Y_n]$. 
These spaces have been extensively studied in recent years. Here is a list of some of the results for these spaces \cite{BGHT, Haim}. For all $n\ge 1$,

\begin{enumerate}
\item Let $E_k=\sum_{i=1}^n y_i^k \partial_{x_i}$. Then 
  $$DH_n=\big\{ P(\partial,E) \Delta_n : P\in  {\mathbb C}[ X_n;Y_n] \big\},$$
where $P(\partial,E)$ denote the operator we get by setting the variables $x_k=\partial_{x_k}$ and $y_k=E_k$ for $1\le k\le n$.

\item  The dimension of $DH_n$  is $(n-1)^{n+1}$;

\item The space $DH_n$ is the so-called parking function
representation of the symmetric group $S_n$;

\item As an $S_n$-module,
 $$DH_n = \bigoplus_{a,b} DH_n^{(a,b)}$$ is bigraded. The $q,t$-Frobenius characteristic is
  $$ {\mathcal F}_{q,t}(DH_n) = \sum_{a,b} q^at^b char(DH_n^{(a,b)}) = < h_1^n, \nabla e_n >$$
where $< h_1^n, \nabla e_n >$ is described in \cite{BGHT, Haim} and is related to Macdonald symmetric functions $H_\lambda(X;q,t)$ of \cite{Ma}.
\end{enumerate}

There are still many open problems regarding $DH_n$. In particular it is not known how to construct an explicit  linear basis.

\end{example}

\begin{example}\label{ex:qsym}
 {\sl Quasi-symmetric functions}:  The Hopf algebra of Quasi-symmetric functions $QSym$ plays a central role in the theory of combinatorial Hopf algebras in \cite{ABS}. It is natural to study it in the context of combinatorial inverse system. Less is known about this system but some remarkable results have been obtained. 

A composition $\alpha\models n$ is a sequence of non-zero positive integers $\alpha=(\alpha_1,\alpha_2,\ldots,$ $\alpha_k)$ where $k\ge 0$ and $n=\alpha_1+\alpha_2+\cdots+\alpha_k$. For $n=0$ there is a unique composition $\alpha=()$, the empty composition. The Hopf algebra of quasi-symmetric functions $QSym$ is the linear span of $\{M_\alpha:\alpha\models n\ge 0\}$. The multiplication is given by the quasi-shuffle
  $$M_\alpha M_\beta =\sum_{\gamma\in \alpha \widetilde{\shuf} \beta} M_\gamma,$$
where for $\alpha=(\alpha_1,\alpha_2,\ldots,\alpha_k)$ and  $\beta=(\beta_1, \beta_2,\ldots, \beta_k)$ we define  $\alpha \widetilde{\shuf} \beta$ recursively as follows.
If $k=0$ or $\ell=0$ then $\alpha \widetilde{\shuf} \beta=\alpha\cdot\beta$ where $\cdot$ denotes the concatenation of lists.
If not,
  $$\alpha \widetilde{\shuf} \beta = \alpha_1\cdot\big((\alpha_2,\ldots,\alpha_k) \widetilde{\shuf} \beta\big) + \beta_1\cdot\big(\alpha \widetilde{\shuf} (\beta_2,\ldots, \beta_k)\big)$$
  $$\quad\qquad + (\alpha_1+\beta_1)\cdot  ((\alpha_2,\ldots,\alpha_k) \widetilde{\shuf} (\beta_2,\ldots, \beta_k)\big).
  $$
  The notation $\gamma\in \alpha \widetilde{\shuf} \beta$ indicates that $\gamma$ is in the support of $\alpha \widetilde{\shuf} \beta$.
The comultiplication is defined by
 $$\Delta(M_\alpha)=\sum_{\beta\cdot\gamma=\alpha} M_\beta \otimes M_\gamma.$$
 The unit is $M_{()}$ and the counit is $\epsilon\colon\QSym\to\mathbb{C}$ where $\epsilon(f)$ is the coefficient of $M_{()}$ in $f$.
 This space is graded by the size of $\alpha$. That is, the degree of $M_\alpha$ is $n$ when $\alpha\models n$.
Again there is an embeding 
$QSym\hookrightarrow {\mathbb C}[\![ X]\!]$ with $X=x_1,x_2,\ldots$ given by 
  $$M_\alpha=\sum_{i_1<i_2<\cdots < i_k} x_{i_1}^{\alpha_1} x_{i_2}^{\alpha_2}\cdots  x_{i_k}^{\alpha_k}.$$
Taking $X_n=x_1,\ldots,x_n$ then $QSym$ is a combinatorial Hopf algebra satisfying the criteria above and we can construct its combinatorial inverse system. The inverse system corresponding to the ideals $I_n=\langle M_\alpha(X_n) : \alpha\models n\ge 1\rangle \subseteq  {\mathbb C}[ X_n]$ is denoted by $SH_n=\{I_n^\perp\}_{n\ge 1}$. 
These spaces have not been extensively studied. Here is a list of some of the partial results and open problems for these spaces \cite{ABB, AB}. For all $n\ge 1$,

\begin{enumerate}
\item We conjecture that 
  $$SH_n=\big\{ P(\partial)Q: P\in  {\mathbb C}[ X_n], \  Q\in\hbox{Soc}(SH_n) \big\},$$
for some set $\hbox{Soc}(SH_n)$ of cardinality $C_{n-1}$, the $(n-1)$-th Catalan number;

\item  The dimension of $SH_n$  is $C_n$ the $n$-th Catalan number;

\item The space $SH_n$ seems to be related to the Temperley-Lieb algebre $TL_n$, yet it is not clear whether this algebra acts on $SH_n$.

\end{enumerate}

\end{example}

These three examples have very rich combinatorial results. More (commutative) combinatorial inverse systems are interesting to study and can be found in the literature \cite{ABB2, St}. 

%%%%%%%%%%%%%%%%%%%%%%%%%%%%%%%%%%%
%%%%%%%%%%%%%%%%%%%%%%%%%%%%%%%%%%%

\section{Symmetric Functions in Non-commuting Variables} \label{sec:SymIS}

To define a non-commutative combinatorial inverse system, we start with a non-commutative combinatorial Hopf algebra $\mathcal{H}$.
Let $\{b_\lambda\}$ be a homogeneous basis for $\mathcal H$ and assume that there is a realization of  ${\mathcal H}$ as a subalgebra of ${\mathbb C}\langle\! \langle X \rangle\! \rangle$, the homogeneous series in countably many non-commuting variables $X$. 
Given this, we can restrict $\mathcal H$ to finitely many variables $X_n$ using an evaluation map ${\mathcal H}\hookrightarrow{\mathbb C}\langle\! \langle X \rangle\! \rangle\to{\mathbb C} \langle X_n \rangle $ where $x=0$ for  all $x\in X- X_n$.
Here ${\mathbb C} \langle X_n \rangle$ is the free associative algebra finitely generated by $X_n$.
 If $X_1\subset X_2\subset \cdots$ and $\lim_{n\to\infty} X_n=X$,
we obtain  a family of ideals $I_n=\langle b_\lambda(X_n) \rangle \subseteq {\mathbb C}\langle X_n\rangle$. The family $\{I_n^\perp\}_{n\ge 0}$ obtained in this way is a {\sl non-commutative combinatorial inverse system}.

\subsection{The combinatorial inverse system $Har_n$}
We trust that our motivating examples in the commutative case convince the reader that combinatorial inverse system are interesting objects to study. We now want to look at the non-commutative anologues of Example~\ref{ex:sym} and Example~\ref{ex:qsym}.
We start with the combinatorial Hopf algebra $NCSym$ of symmetric functions in non-commuting variables \cite{BRRZ,RS,W}. It is simpler to first describe the embeding $NCSym\to {\mathbb C}\langle\! \langle X \rangle\! \rangle$ for $X=x_1,x_2,\ldots$.

A {\sl set
partition} $\Phi$ of $k$ is a set of nonempty subsets
$\Phi_1,\Phi_2,\ldots,\Phi_n\subseteq [k]=\{1,2,\ldots, k\}$ such
that $\Phi_i\cap \Phi_j=\{\}$ for $i\neq j$ and $\Phi_1\cup
\Phi_2\cup \cdots\cup \Phi_n=[k]$. We indicate that $\Phi$ is a set
partition of $k$ by $\Phi\vdash[k]$. The subsets $\Phi_i$ are called
the {\sl parts} of the set partition. The number of nonempty parts is
referred to as the {\sl length} and is denoted by $\ell(\Phi)$. 
A {\sl monomial} with variables in $X$ is a word
$w=w_1w_2\cdots w_k$ with $w_i\in X$. The word $w$ can be viewed as a function
$w:[k]\rightarrow X$. Let $\nabla(w)=\{w^{-1}(x):x\in X\}\backslash
\{\emptyset\}$. Clearly $\nabla(w)$ is a set partition of $[k]$. For
$\Phi\vdash[k]$ we define
$$M_\Phi(X)=\sum_{\nabla(w)=\Phi}w,$$ where the sum is over all $w$
whose corresponding set partition is $\Phi$. For the empty set
partition, we define by convention $M_{\{\}}=1$. For example, when $k=4$ and $\Phi={13.2.4}:=\{\{1,3\},\{2\},\{4\}\}$,
$$M_{13.2.4}=x_1x_2x_1x_3+x_1x_3x_1x_2+x_2x_1x_2x_3+x_2x_3x_2x_1+x_3x_1x_3x_2+x_3x_2x_3x_1 + \cdots .$$
The vector space $NCSym$ has a linear basis
$\{M_\Phi(X)\}_{k\geq0,\Phi\vdash[k]}.$ Its multiplication is given by the multiplication in ${\mathbb C}\langle\! \langle X \rangle\! \rangle$.
Its comultiplication is given by
$$\Delta(M_\Phi) = \sum_{\Psi\cup\Xi=\Phi} M_\Psi \otimes M_\Xi.$$

Let  $X_n=x_1,\ldots,x_n$ and let 
$$I_n=\langle M_\Phi(X_n) : \Phi\vdash [k], k>0\rangle.$$ 
We are interested in computing $Har_n=I_n^{\perp}$. The space $NCSym_n$ linearly spanned by $\{M_\Phi(X_n) : \Phi\vdash [k], k\ge 0\}$ is
known to correspond to the invariants of the symmetric group $S_n$ in non-commutative variables $X_n$~\cite{BRRZ,RS,W}.
{}From Lemma \ref{lem:eq} in non-commuting variables $\varphi\in Har_n\Leftrightarrow
\{\overleftarrow{M_\Phi}(d)d_u\cdot\varphi=0\}_{\Phi\neq\{\}}$. This is a
non-commutative system of (differential) equations containing the equation $\sum d_{x_i}^2 \varphi =0$. For this reason we say that
$Har_n$ is a \textit{non-commuting Harmonic system} related to the symmetric group. There are other non-commutative Harmonic systems related to the symmetric group studied in~\cite{BRRZ,BRZ} but they are much larger.

\begin{corollary}
For a word $w\in {\mathbb C}\langle X_n\rangle$, define
\begin{equation}\Delta_w\stackrel{def}{=}\sum_{\sigma\in S_n}
\sum_{\pi\in S_k}(-1)^{\ell(\sigma)}\sigma\circ
w\circ\pi.
\end{equation}
Then $\Delta_w\in Har_n$ if and only if $\Delta_w=0$
or $\chi(w)=x^{\rho}$, where $\rho=(n-1,n-2,\ldots,1,0)$.
\end{corollary}

\begin{proof} It is easy to see that $\chi(I_n)=\langle p_k(X_n)\rangle$ the ideal considered in Example~\ref{ex:sym}. Hence $\chi(I_n)^\perp=H_n$. Let
$$A_{\alpha}\stackrel{def}{=}\sum_{\sigma\in S_n}(-1)^{\ell(\sigma)}\sigma\circ x^{\alpha}.$$
Then $\Delta_w=\psi(A_{\alpha})$ when $\chi(w)=x^{\alpha}$.
Using Corollary \ref{cor:inc},
$$\Delta_w\in Har_n\Leftrightarrow A_{\alpha}\in H_n.$$
Since $H_n$ is isomorphic to the left regular
representation of $S_n$ (see Example~\ref{ex:sym}), there is only one
occurrence of the sign representation (obtained by taking
$A_{\rho}$). Hence $A_{\alpha}\in H_n \Leftrightarrow
\alpha=\rho$ or $A_{\alpha}=0$.
\end{proof}

Thus $d_u\Delta_w\in Har_n$ for all $u$ and
$\chi(w)=x^\rho$. This gives us a copy of the left regular
representation in $Har_n$. For $n=1$ and $2$, this is all of
$Har_n$. For $n=3$, we see more.
We have computed $Har_n$ for $n=1,2,3,4$ and the following
surprising fact arises: $\dim Har_n=1,2,9,946,\ldots$.
The Hilbert series of $Har_n$ is
$$\hbox{Hilb}_{Har_n}(t)=\sum_{d=0}^{\infty}\dim(Har_n^{(d)})t^d,$$
where $Har_n^{(d)}$ is the homogeneous component of degree $d$ in
$Har_n$.
We list the following data obtained using computers:
\begin{itemize}
    \item $n=0,\ \hbox{Hilb}_{Har_0}(t)=1,$
    \item $n=1,\ \hbox{Hilb}_{Har_1}(t)=1,$
    \item $n=2,\ \hbox{Hilb}_{Har_2}(t)=1+t,$
    \item $n=3,\ \hbox{Hilb}_{Har_3}(t)=1+2t+3t^2+3t^3,$
    \item $n=4,\ \hbox{Hilb}_{Har_4}(t)=1+ 3t + 8t^2 + 20t^3 + 47t^4 + 102t^5 + 197t^6 +
                                        308t^7 +\hbox{\hspace{3.52cm}}248t^8 + 12t^9 $,
    \item $n=5,\ \hbox{Hilb}_{Har_5}(t)=1 + 4t + 15t^2 + 55t^3 + 199t^4 + 712t^5 + 2520t^6+\cdots $
\end{itemize}

\begin{conj}\label{co:Har} The non-commuting Harmonic system $Har_n$ is finite dimensional for any $n$.
\end{conj}

This conjecture would imply that $I_n$ has a decidable Gr\"obner basis for all $n$. This is a striking fact on its own. Remark that in contrast, the Harmonic systems computed in~\cite{BRRZ,BRZ} are infinite.

\subsection{The subspace of alternating elements}

In the commutative examples it turned out that the  multiplicity of the alternating representation has very combinatorial behavior.
It is only natural to do the same for $Har_n$.
We now study the intersection of the inverse system $Har_n$ with the space of
alternating functions. 

A {\sl set composition} $A$ of a set $[n]$ is a list $A=(A_1,A_2,\ldots,A_k)$ such that $A_i\ne \emptyset$ and $\{A_1,A_2,\ldots,A_k\}\vdash[n]$. We denote this by $A\models[n]$.
A function $f$ in $\mathbb{C}\langle X_n\rangle$ is {\sl alternating}  if $\sigma\circ f=(-1)^{\ell(\sigma)}f$, for all $\sigma \in S_n $. Let $Alt_n$ be the set of  all
alternating functions in $n$ non-commutating variables. We will construct a basis of $Alt_n$ which is given by $\{ \A_{\Phi}\}$ indexed by set compositions
$\Phi=(\Phi_1,\Phi_2,\ldots,\Phi_k)$  with
$\min\Phi_i<\min\Phi_{i+1}$
 and $n-1\leq k\leq n$. These set compositions may be identified with set partitions of size $n-1\leq k\leq n$ where the parts are given in the prescribed order.
 
Conjecture~\ref{co:Har} would imply that
$$deg(Har_n^{(d)}\cap Alt^{(d)}_n)=0$$
when $d$ is sufficiently large, where $Alt_n^{(d)}$ and $Har_n^{(d)}$ denote the respective homogeneous component of degree $d$ in
in each space.

A \textit{generalized set composition} $A$ of $[d]$ is a list of subsets
$(A_1,A_2,\ldots,A_k)$, where $A_i\subseteq[d]=\{1,2,\ldots,d\}$ such that
$A_i\cap A_j=\emptyset$ for $i\neq j$ and
$A_1\cup A_2\cup\ldots\cup A_k=[d]$ (some parts may be empty). There is a one-to-one
correspondence between words and generalized set compositions
$w\leftrightarrow  A$. Since we use generalized set compositions
of $d$ to decide the positions of variables in the corresponding
words of degree $d$, we assume $k=n$ the number of variables.
For example, let $n=3$ and $w=x_1x_3x_1$. Then $d=3$, the degree of $w$, and
$ A=(\{1,3\},\emptyset,\{2\})$. For convenience, we write $ A=13.\emptyset.2$.

Given a generalized set composition $A=(A_1,A_2,\ldots,A_n)\models [d]$ we define an alternating function
$$\A_{A}=\sum_{\sigma\in S_n}(-1)^{\ell(\sigma)}\sigma\circ A,$$
where in the sum $A$ is the corresponding word. For example, let $d=3$, 
$n=3$ and $A=13.2.\emptyset$. Then
$$\A_A=x_1x_2x_1-x_2x_1x_2-x_3x_2x_3-x_1x_3x_1+x_2x_3x_2+x_3x_1x_3.$$
Two properties hold:

(1) $\A_A=(-1)^{\ell(\pi)}\A_{\pi\circ A}$ for all $\pi\in
S_n$, where
$\pi\circ A=(A_{\pi(1)},A_{\pi(2)},\ldots,A_{\pi(n)})$;

(2) $\A_A=0$ if $\#\{A_i: A_i=\emptyset\}\geq2$;

\noindent
{}From (1) and (2) above, it is clear that any $\A_A =\pm \A_\Phi$  for $\Phi=(\Phi_1,\Phi_2,\ldots,\Phi_n)$ a
generalized set composition  where at most $\Phi_n=\emptyset$ and the non-empty parts are ordered by
$\min\Phi_i<\min\Phi_{i+1}$.  For $\Phi$ as described, each $\A_\Phi$  has distinct support, hence $\{\A_\Phi : \Phi\models [d] \hbox{ generalized, ordered, }\Phi_{n-1}\ne \emptyset\}$
is a basis of $Alt_n^{(d)}$.

For any generalized set composition
$A=(A_1,A_2,\ldots,A_n)$ and a given  coloring function
$\epsilon:[n]\rightarrow\{\pm1\}$, we get a two-colored generalized
set composition
$$A^{\epsilon}=(A_1^{\epsilon_1},A_2^{\epsilon_2},\ldots,A_n^{\epsilon_n}),$$
where $\epsilon_i=\epsilon(i)$. Let $T=\epsilon^{-1}(1)$. Define
$$\A_{A^{\epsilon}}=\sum_{\sigma\in S_{T}}(-1)^{\ell(\sigma)}\sigma\circ A^{\epsilon},$$
where in the sum $A^{\epsilon}$ is the word corresponding to
$A$ and $\sigma$ permutes the indices given by $T$ while fixing the
other parts. For convenience, we write
$2.\overline{13}.4.\overline{\emptyset}$ instead of
$2^1.13^{-1}.4^1.\emptyset^{-1}$. Then
$$\A_{2.\overline{13}.4.\overline{\emptyset}}=x_2x_1x_2x_3-x_2x_3x_2x_1.$$

\begin{proposition} \label{prop:alt}
Let $\Psi\vdash [k]$ be a  set partition,
$A\models[\ell]$ be a generalized set composition
 and $\Phi\models[d]$ be an ordered generalized set composition with at most one $\emptyset$.
 Then $\overleftarrow{M_{\Psi}}(d)d_{\overleftarrow{A}}\cdot \A_{\Phi}$ is $0$ or
$\pm \A_{\Theta^{\epsilon}}$ for some two-colored generalized set
composition $\Theta^{\epsilon}$.
\end{proposition}

\begin{proof} We prove it case by case.
Suppose that 
\begin{equation}\label{eq:cap}
 A_1=A_1\cap\Phi_{i_1},\quad
 A_2= A_2\cap\Phi_{i_2},\quad\ldots,\quad  A_n= A_n\cap\Phi_{i_n}
 \end{equation}
  for
some distinct $i_1,i_2,\ldots, i_n$. Hence, there are some terms in
$\A_{\Phi}$ beginning with the word corresponding to $ A$. Choose a
fixed sequence $i_1,i_2,\ldots, i_n$ and reorder $\Phi$ as
$\Phi^{\tau}=(\Phi_{i_1},\Phi_{i_2},\ldots,\Phi_{i_n})$, where
$\tau$ is the permutation corresponding to this sequence. Then there
is one word in the remaining terms in $d_{\overleftarrow{ A}}\cdot
\A_{\Phi}$ corresponding to $st(\Phi^{\tau}\setminus A)$, where
$st(\Phi^{\tau}\setminus A)$ means deleting the numbers in $ A$
from $\Phi^\tau$ then standardizing it (that is, subtracting $\ell$ from all the numbers in $\Phi^{\tau}\setminus A$). Define a two-colored generalized
set composition $st(\Phi^{\tau}\setminus A)^{\epsilon}$ by
$\epsilon_i=-1$ if $ A_i\neq\emptyset$ and $1$ otherwise. Then
$d_{\overleftarrow{ A}}\cdot
\A_{\Phi}=(-1)^{\ell(\tau)}\A_{st(\Phi^{\tau}\setminus A)^{\epsilon}}$.
If Equation~(\ref{eq:cap}) is not satisfied, then there is no term in
$\A_{\Phi}$ beginning with the word corresponding to $A$.
Therefore, $d_{\overleftarrow{ A}}\cdot \A_{\Phi}=0$.

Hence, we may assume that Equation~(\ref{eq:cap}) is satisfied and let $\Gamma=st(\Phi^{\tau}\setminus A)$. Then
$\overleftarrow{M_{\Psi}}(d)d_{\overleftarrow{ A}}\cdot
\A_{\Phi}=(-1)^{\ell(\tau)}\overleftarrow{M_{\Psi}}(d)\cdot
\A_{\Gamma^{\epsilon}}$. 
The set partition $\Psi=\{\Psi_1,\Psi_2,\ldots,\Psi_r\}$ for some $r\le n$.
Suppose 
\begin{equation}\label{eq:cap2}
 \Psi_1=\Psi_1\cap \Gamma_{i_1},\quad
 \Psi_2= \Psi_2\cap \Gamma_{i_2},\quad\ldots,\quad  \Psi_r= \Psi_r\cap \Gamma_{i_n}
 \end{equation}
  for
some distinct $i_1,i_2,\ldots, i_r$. Since $\Psi$ is a set, we can reorder the $\Psi_i$ as needed. For any monomial in $M_\Psi$ there is a unique monomial in $\A_{\Gamma^{\epsilon}}$ with the appropriate sign which gives
$\overleftarrow{M_{\Psi}}(d)\cdot
\A_{\Gamma^{\epsilon}}=\A_{st(\Gamma\setminus\Psi)^{\epsilon}}.$
If Equation~(\ref{eq:cap2}) is not satisfied, then
there is no term in $\A_{st(\Phi^{\tau}\setminus A)^{\epsilon}}$
beginning with the word $\sigma\circ\Psi$ for all $\sigma\in
S_n/S_L$. Hence $\overleftarrow{M_{\Psi}}(d)\cdot
\A_{st(\Phi^{\tau}\setminus A)^{\epsilon}}=0$.

In all cases, if any of Equation~(\ref{eq:cap}) or Equation~(\ref{eq:cap2}) is not satisfied, then
$\overleftarrow{M_{\Psi}}(d)d_{\overleftarrow{ A}}\cdot
\A_{\Phi}=0$. Otherwise
$$\overleftarrow{M_{\Psi}}(d)d_{\overleftarrow{ A}}\cdot
\A_{\Phi}=      (-1)^{\ell(\tau)}
\A_{st(st(\Phi^{\tau}\setminus A)\setminus\Psi)^{\epsilon}}\, .
$$
We let $\Theta=st(st(\Phi^{\tau}\setminus A)\setminus\Psi)$ and the proof is completed.
\end{proof}

We remark that in the proof of Proposition~\ref{prop:alt} we could reorder $A$ so that $\tau$ is the identity. That is, we can assume that $A$ is an ordered generalized composition. 

If we want to compute $Har_n^{(d)}\cap Alt^{(d)}_n$, we need to solve the equations
  $$ \overleftarrow{M_{\Psi}}(d)d_{\overleftarrow{A}}\cdot P=0,\  \forall  \Psi, A$$
  and $P\in Alt_n^{(d)}$. It is easy to see that for fixed $A$ and $\Psi$, the possible $\Theta^\epsilon$ are all linearly independent. This gives us a system of linear equations, 
  \begin{equation}\label{eq:se}
   [\overleftarrow{M_{\Psi}}(d)d_{\overleftarrow{A}}\cdot P]_{\A_{\Theta^\epsilon}} =0
 \end{equation}
 where $[f]_{\A_{\Theta^\epsilon}} $ denotes the coefficient of ${\A_{\Theta^\epsilon}}$ in $f$ and
  $$ P=\sum_{\Phi\models [d]\  ordered} c_\Phi\A_\Phi.$$
  The system~(\ref{eq:se}) is explicit and may be easier to handle than the full space $Har_n$. As a weaker conjecture we propose
  
  \begin{conj} The solution for the linear equations in the System~(\ref{eq:se}) is all of $Alt_n^{(d)}$ for large $d$. \end{conj}
  
  This system of linear equations can be generated easily by computer and we have checked that the conjecture holds up to $n=5$.
  There is also more structure in this system and we are hopeful that after a few well organized elementary transformations one can obtain an upper triangular system. 
   
%%%%%%%%%%%%%%%%%%%%%%%%%%%%%%%%%%%
%%%%%%%%%%%%%%%%%%%%%%%%%%%%%%%%%%%

\section{Quasi-symmetric functions in non-commutative variables}\label{sec:QSymIS}

A second example of a non-commutative combinatorial inverse system is given by quasi-symmetric functions in non-commutative variables~\cite{BZ}. We have recently shown~\cite{AB2} that the associated inverse systems are finite. The proof does not give the quotient,
which is still an open problem. However, we present the result here as further evidence that combinatorial inverse systems are special and as support for the conjectures in Section~\ref{sec:SymIS}.

The combinatorial Hopf algebra $NCQSym$ of quasi-symmetric functions in non-commutative variables is the vector space freely generated by the set $\{\mathcal{M}_A: A\models[d],d\ge 0\}$. Again we let $()$ be the unique set composition of $[0]$.
The multiplication is given by  the quasi-shuffle of set compositions. That is, for $A=(A_1,A_2,\ldots,A_k)\models[d]$ and $B=(B_1,B_2,\ldots, B_\ell)\models [q]$ we define 
$A \widetilde{\shuf} B^{\uparrow d}$ recursively. The notation $B^{\uparrow d}$ means that we add $d$ to each entry of $B$.
If $k=0$ or $\ell=0$ then $A \widetilde{\shuf}B^{\uparrow d}=A\cdot B$ where $\cdot$ denotes the concatenation of lists.
If not,
  $$A \widetilde{\shuf} B^{\uparrow d} = A_1\cdot\big((A_2,\ldots,A_k) \widetilde{\shuf} B^{\uparrow d}\big) + B^{\uparrow d}_1\cdot\big(A \widetilde{\shuf} (B^{\uparrow d}_2,\ldots, B^{\uparrow d}_k)\big)$$
  $$\quad\qquad + (A_1\cup B^{\uparrow d}_1)\cdot  ((A_2,\ldots,A_k) \widetilde{\shuf} (B^{\uparrow d}_2,\ldots, B^{\uparrow d}_k)\big).
  $$
We then define
  $$\mathcal{M}_A \mathcal{M}_B =\sum_{C\in A \widetilde{\shuf} B^{\uparrow d}} \mathcal{M}_C$$
where $C\in A \widetilde{\shuf} B^{\uparrow d}$  indicates that $C$ is in the support of $A \widetilde{\shuf} B^{\uparrow d}$.
The comultiplication is defined by
 $$\Delta(\mathcal{M}_A)=\sum_{B\cdot C=A} \mathcal{M}_{st(B)} \otimes \mathcal{M}_{st(C)}.$$
Here, $st(-)$ is the standardization map defined as follows. For $B\cdot C=A\models [d]$, we have that $B\models S$ and $C\models T$ where $S\cup T=[d]$ and $S\cap T=\emptyset$.  For any $S\subseteq[d]$, there is a unique order preserving map $\phi\colon \{1,2,\ldots, |S|\} \to S$. For any $B\models S$, we let $st(B)\models \{1,2,\ldots, |S|\}$ be the unique set composition obtained using $\phi$, in other words $st(B)=\phi^* B$. 
The space $NCQSym$ is graded by $\deg(\mathcal{\mathcal{M}}_A)=d$ when $A\models[d]$. 

As seen in \cite{BZ}, there is  an embeding 
$NCQSym\hookrightarrow {\mathbb C}\langle\!\langle X\rangle\!\rangle$ with $X=x_1,x_2,\ldots$ given by 
  $$\mathcal{M}_A=\sum_{\tilde\nabla(w)=A} w,$$
where $\tilde\nabla(w)=( w^{-1}(x_i) )_{i=1}^\infty \setminus \emptyset$. In other words, the infinite sequence $( w^{-1}(x_i) )_{i=1}^\infty $ has only finitely many non-empty parts. After removing the empty parts, we obtain a set composition denoted by  $\tilde\nabla(w)$.

Taking $X_n=x_1,\ldots,x_n$  then $NCQSym$ is a combinatorial Hopf algebra satisfying the criteria above and we can construct its combinatorial inverse system. This is the inverse system $SHar_n=\{I_n^\perp\}_{n\ge 1}$ corresponding to the ideals $I_n=\langle \mathcal{M}_A(X_n) : A\models [d]\ge 1\rangle \subseteq  {\mathbb C}\langle X_n\rangle$. 
These spaces have not been studied extensively. The only known theorem is the following:

\begin{theorem}
$\dim(SHar_n)<\infty$.
\end{theorem}

This is shown in \cite{AB2}. In parallel with the commutative case, we believe that the non-commutative combinatorial inverse systems will prove themselves to
be very rich objects of study. It would not be surprising to discover that they all have very special Gr\"obner basis.

%%%%%%%%%%%%%%%%%%%%%%%%%%%%%%%%%%%

\end{document}